\documentclass{amsart}         
\usepackage{breqn,amssymb,amsthm}  

\newtheorem{theorem}{Theorem}

\theoremstyle{definition}

\theoremstyle{remark}

\numberwithin{equation}{section}

\begin{document}

\title[A note on an integral by Grigorii Mikhailovich Fichtenholz]{A note on an integral by Grigorii Mikhailovich Fichtenholz}

\author{Robert Reynolds}
\address{Department of Mathematics and Statistics, York University, Faculty of Science, York University}
\email{milver@my.yorku.ca}




\maketitle

\begin{abstract}

In this manuscript, the author derive a definite integral involving the logarithmic function, function of powers and polynomials in terms of the Lerch function. A summary of the results is produced in the form of a table of definite integrals for easy referencing by readers.

\end{abstract}

\section{Introduction}

In 1948 Grigorii Mikhailovich Fichtenholz produced volume II of his three volume collection \cite{fitch}. In volume II the author found an integral (3.244.4) in \cite{grad} which is of interest because of its closed form solution over the real line. However, upon closer inspection and evaluation of this integral and applying our simultaneous contour integral method we found this integral is not symmetric over the real when the logarithmic function is introduced into the integrand. This logarithmic term appears after applying our contour integral method to this Fichtenholz integral. This consequence lead us to produce this manuscript which achieves two objectives. The first, is that of producing formal derivations for some definite integrals in Table 3.244 in \cite{grad}. The second goal is to produce more definite integrals as an expansion of the current Table 3.244 in \cite{grad}. These gaols are achieved by using this integral by Fichtenholz along with our contour integral method to form a closed solution in terms of the Lerch function. The Lerch function being a special function has the fundamental property of analytic continuation, which enables us to widen the range of evaluation for the parameters involved in our definite integral.  

The definite integral the author derived using the integral by Fichtenholz in this manuscript is given by
 
 \begin{dmath}
 \int_{0}^{\infty}\frac{\left(x^{2 n}-x^{2 m}\right) \log ^k\left(a x^2\right)}{x^{2 l}-1}dx
  \end{dmath}

in terms of the Lerch function, where the parameters $k$, $a$, $m$, $n$ and $l$ are general complex numbers. A summary of the results is given in a table of integrals for easy reading. This work is important because the author were unable to find similar results in current literature. Tables of definite integrals provide a useful summary and reference for readers seeking such integrals for potential use in their research. We use our simultaneous contour integration method to aid in our derivations of the closed forms solutions in terms of the Lerch function, which provides analytic continuation of the results. The derivations follow the method used by us in \cite{reyn2}. The generalized Cauchy's integral formula is given by

\begin{equation}\label{intro:cauchy}
\frac{y^k}{k!}=\frac{1}{2\pi i}\int_{C}\frac{e^{wy}}{w^{k+1}}dw.
\end{equation}

 This method involves using a form of equation (\ref{intro:cauchy}) then multiply both sides by a function, then take a definite integral of both sides. This yields a definite integral in terms of a contour integral. A second contour integral is derived by multiplying equation (\ref{intro:cauchy}) by a function and performing some substitutions so that the contour integrals are the same.
 

\section{Definite integral of the contour integral}

We use the method in \cite{reyn2}. Here the contour is similar to Figure 2 in \cite{reyn2}. Using a generalization of Cauchy's integral formula equation (\ref{intro:cauchy}) replace $y$ by $\log(ax^2)$ followed by multiplying both sides by $\frac{x^{2 m}-x^{2 n}}{1-x^{2 l}}$ and taking the definite integral over $x\in[0,\infty)$ to get

\begin{dmath}\label{eq1a}
\frac{1}{k!}\int_{0}^{\infty}\frac{\left(x^{2 n}-x^{2 m}\right) \log ^k\left(a x^2\right)}{x^{2 l}-1}dx
=\frac{1}{2\pi i}\int_{0}^{\infty}\int_{C}\frac{a^w w^{-k-1} \left(x^{2 (m+w)}-x^{2 (n+w)}\right)}{1-x^{2 l}}dwdx
=\frac{1}{2\pi i}\int_{C}\int_{0}^{\infty}\frac{a^w w^{-k-1} \left(x^{2 (m+w)}-x^{2 (n+w)}\right)}{1-x^{2 l}}dxdw
=\frac{1}{2\pi i}\int_{C}\frac{\pi  a^w w^{-k-1} \left(\cot \left(\frac{\pi  (2 (m+w)+1)}{2
   l}\right)-\cot \left(\frac{\pi  (2 (n+w)+1)}{2 l}\right)\right)}{2 l}dw
\end{dmath}

from equation (3.244.4) in \cite{grad} and the integral is valid for $a$, $m$, $l$, $n$ and $k$ complex and $-1<Re(w+m)<0$ an $-1<Re(w+n)<0$. The logarithmic function is defined in equation (4.1.2) in \cite{as}

\section{The Lerch function}

The Lerch function has a series representation given by
\begin{equation}\label{lerch:eq}
\Phi(z,s,v)=\sum_{n=0}^{\infty}(v+n)^{-s}z^{n}
\end{equation}

where $|z|<1, v \neq 0,-1,..$ and is continued analytically by its integral representation given by

\begin{equation}\label{armenia:eq8}
\Phi(z,s,v)=\frac{1}{\Gamma(s)}\int_{0}^{\infty}\frac{t^{s-1}e^{-vt}}{1-ze^{-t}}dt=\frac{1}{\Gamma(s)}\int_{0}^{\infty}\frac{t^{s-1}e^{-(v-1)t}}{e^{t}-z}dt
\end{equation}

where $Re(v)>0$, or $|z| \leq 1, z \neq 1, Re(s)>0$, or $z=1, Re(s)>1$.  

\section{Infinite sum of the contour integral}

\subsection{Derivation of the first contour integral}

In this section we will derive the contour integral given by

\begin{dmath}
\frac{1}{2\pi i}\int_{C}\frac{\pi  a^w w^{-k-1} \cot \left(\frac{\pi  (2 (m+w)+1)}{2 l}\right)}{2 l}dw
\end{dmath}

Again, using the method in \cite{reyn2} and equation (\ref{intro:cauchy}), we replace $y$ by $\log (a)+\frac{2 i \pi  (y+1)}{l}$ multiply both sides by $-\frac{2 i \pi  }{l}e^{\frac{i \pi  (2 m+1) (y+1)}{l}}$ and take the infinite sum of both sides over $y \in [0,\infty)$ simplifying in terms the Lerch function to get

\begin{dmath}\label{eq2a}
-\frac{i (2 \pi )^{k+1} \left(\frac{i}{l}\right)^k e^{\frac{i (2 \pi  m+\pi )}{l}} \Phi \left(e^{\frac{i (2 m+1) \pi }{l}},-k,1-\frac{i l
   \log (a)}{2 \pi }\right)}{l k!}
   =-\frac{1}{2\pi i}\sum_{y=0}^{\infty}\int_{C}\frac{2 i \pi  a^w w^{-k-1} e^{\frac{i \pi  (y+1) (2 m+2 w+1)}{l}}}{l}dw
    =-\frac{1}{2\pi i}\int_{C}\sum_{y=0}^{\infty}\frac{2 i \pi  a^w w^{-k-1} e^{\frac{i \pi  (y+1) (2 m+2 w+1)}{l}}}{l}dw
   =\frac{1}{2\pi i}\int_{C}\left(\frac{\pi  a^w w^{-k-1} \cot \left(\frac{\pi  (2 m+2 w+1)}{2 l}\right)}{l}+\frac{i \pi  a^w w^{-k-1}}{l}\right)dw
\end{dmath}

similar to equation (1.232.1) in \cite{grad} where

\begin{dmath}
\cot(x)=-2i\sum_{y=0}^{\infty}e^{2xi(y+1)}-i
\end{dmath}

where $Re(x)>0$.

\subsection{Derivation of the second contour integral}

In this section we will derive the contour integral given by

\begin{dmath}
-\frac{1}{2\pi i}\int_{C}\frac{\pi  a^w w^{-k-1} \cot \left(\frac{\pi  (2 n+2 w+1)}{2 l}\right)}{l}dw
\end{dmath}

Again, using the method in \cite{reyn2} and equation (\ref{intro:cauchy}), we replace $y$ by $\log (a)+\frac{2 i \pi  (y+1)}{l}$ multiply both sides by $-\frac{2 i \pi }{l} e^{\frac{i \pi  (2 n+1) (y+1)}{l}}$ and take the infinite sum of both sides over $y \in [0,\infty)$ simplifying in terms the Lerch function to get

\begin{dmath}\label{eq2b}
\frac{i (2 \pi )^{k+1} \left(\frac{i}{l}\right)^k e^{\frac{i (2 \pi  n+\pi )}{l}} \Phi \left(e^{\frac{i (2 n+1) \pi }{l}},-k,1-\frac{i l
   \log (a)}{2 \pi }\right)}{l k!}
   =\frac{1}{2\pi i}\sum_{y=0}^{\infty}\int_{C}\frac{2 i \pi  a^w w^{-k-1} e^{\frac{i \pi  (y+1) (2 n+2 w+1)}{l}}}{l}dw
   =\frac{1}{2\pi i}\int_{C}\sum_{y=0}^{\infty}\frac{2 i \pi  a^w w^{-k-1} e^{\frac{i \pi  (y+1) (2 n+2 w+1)}{l}}}{l}dw
   =-\frac{1}{2\pi i}\int_{C}\left(\frac{\pi  a^w w^{-k-1} \cot \left(\frac{\pi  (2 n+2 w+1)}{2 l}\right)}{l}-\frac{i \pi  a^w
   w^{-k-1}}{l}\right)dw
\end{dmath}

similar to equation (1.232.1) in \cite{grad} where

\begin{dmath}
\cot(x)=-2i\sum_{y=0}^{\infty}e^{2xi(y+1)}-i
\end{dmath}

where $Re(x)>0$.

\section{Definite integral in terms of the Lerch function}

\begin{theorem}
For all $m,n,a,k,l \in\mathbb{C}$,
\begin{dmath}\label{eq:dilerch}
\int_{0}^{\infty}\frac{\left(x^{2 n}-x^{2 m}\right) \log ^k\left(a x^2\right)}{x^{2 l}-1}dx=\frac{i 2^k \pi ^{k+1} \left(\frac{i}{l}\right)^k e^{\frac{i (2 \pi
    n+\pi )}{l}} \Phi \left(e^{\frac{i (2 n+1) \pi }{l}},-k,1-\frac{i l \log (a)}{2 \pi }\right)}{l}-\frac{i 2^k \pi ^{k+1}
   \left(\frac{i}{l}\right)^k e^{\frac{i (2 \pi  m+\pi )}{l}} \Phi \left(e^{\frac{i (2 m+1) \pi }{l}},-k,1-\frac{i l \log (a)}{2 \pi
   }\right)}{l}
\end{dmath}
\end{theorem}

\begin{proof}
Since the right-hand side of equation (\ref{eq1a}) is equal to the sum of equations (\ref{eq2a}) and (\ref{eq2b}) we can equate the left-hand sides and simplify the factorials.
\end{proof}


\section{Derivation of a logarithmic integral}

\begin{theorem}
For $Re(m)<Re(l)$ and $Re(n)<Re(l)$, 

\begin{dmath}\label{eq5a}
\int_{0}^{\infty}\frac{\log \left(x^2\right) \left(x^{2 n}-x^{2 m}\right)}{x^{2 l}-1}dx=\frac{\pi ^2 \left(\csc ^2\left(\frac{2 \pi  n+\pi }{2 l}\right)-\csc
   ^2\left(\frac{2 \pi  m+\pi }{2 l}\right)\right)}{2 l^2}
\end{dmath}

\end{theorem}

\begin{proof}
Use equation (\ref{eq:dilerch}) and set $k=a=1$ simplify using entry (2) in Table below (64:12:7) in \cite{atlas}.
\end{proof}

\section{Derivation of entry 3.244.4 in \cite{grad}}

\begin{theorem}
For $Re(m)<Re(l)$ and $Re(n)<Re(l)$, 
\begin{dmath}
\int_{0}^{\infty}\frac{x^{2 n}-x^{2 m}}{x^{2 l}-1}dx
=\frac{\pi  \left(\cot \left(\frac{2 \pi  m+\pi }{2 l}\right)-\cot \left(\frac{2 \pi  n+\pi }{2 l}\right)\right)}{2 l}
\end{dmath}
\end{theorem}

\begin{proof}
Use equation (\ref{eq5a}) and set $k=0$ simplify using entry (2) in Table below (64:12:7) in \cite{atlas}.
\end{proof}

\section{Derivation of entry 4.235.1 in \cite{grad}}

\begin{theorem}
For $Re(n)>1$, 
\begin{dmath}\label{eq5c}
\int_{0}^{\infty}\frac{(x-1) x^{n-2} \log (x)}{x^{2 n}-1}dx=-\frac{\pi ^2 \tan ^2\left(\frac{\pi }{2 n}\right)}{4 n^2}
\end{dmath}

\end{theorem}

\begin{proof}
Use equation (\ref{eq5a}) and replace $n$ by $\frac{n-2}{2}$, $m$ by $\frac{n-1}{2}$ and $l$ by $n$ and simplify.
\end{proof}

\section{Derivation of entry 4.235.2 in \cite{grad}}

\begin{theorem}
For $Re(n)>Re(m)>0$, 
\begin{dmath}
\int_{0}^{\infty}\frac{\left(x^2-1\right) x^{m-1} \log (x)}{x^{2 n}-1}dx=-\frac{\pi ^2 \left(\csc ^2\left(\frac{\pi  m}{2 n}\right)-\csc ^2\left(\frac{\pi 
   (m+2)}{2 n}\right)\right)}{4 n^2}
\end{dmath}
\end{theorem}

\begin{proof}
Use equations (\ref{eq5a}) and replace $n$ by $\frac{m-1}{2}$, $m$ by $\frac{m+1}{2}$ and $l$ by $n$ simplify using equation (64:10:2) in \cite{atlas}.
\end{proof}

\section{Derivation of entry 4.235.3 in \cite{grad}}

\begin{theorem}
For $Re(n)>2$ and $Im(n)>2$,
\begin{dmath}
\int_{0}^{\infty}\frac{\left(x^2-1\right) x^{n-3} \log (x)}{x^{2 n}-1}dx=-\frac{\pi ^2 \tan ^2\left(\frac{\pi }{n}\right)}{4 n^2}
\end{dmath}
\end{theorem}

\begin{proof}
Use equations (\ref{eq5a}) and replace $n$ by $\frac{n-3}{2}$, $m$ by $\frac{n-1}{2}$ and $l$ by $n$ simplify using equation (64:4:2) in \cite{atlas}.
\end{proof}

\section{Definite integral in terms of the Polylogarithm function}

Using equation (\ref{eq:dilerch}) and setting $a=1$ simplifying to get

\begin{dmath}
\int_{0}^{\infty}\frac{\log ^k(x) \left(x^{2 n}-x^{2 m}\right)}{x^{2 l}-1}dx=\frac{\pi ^{k+1} \left(\frac{i}{l}\right)^{k-1}
   \left(\text{Li}_{-k}\left(e^{\frac{i (2 m+1) \pi }{l}}\right)-\text{Li}_{-k}\left(e^{\frac{i (2 n+1) \pi }{l}}\right)\right)}{l^2}
\end{dmath}

from equation (64:12:2) in \cite{atlas}.

\section{Definite integral in terms of the logarithm of trigonometric functions }

\begin{theorem}
For all $m,n,p,q,l \in \mathbb{C}$
\begin{dmath}
\int_{0}^{\infty}\frac{x^{2 m}-x^{2 n}-x^{2 p}+x^{2 q}}{\left(x^{2 l}-1\right) \log (x)}dx=\log \left(\frac{\left(\cos \left(\frac{\pi  (n-p)}{l}\right)-\cos
   \left(\frac{\pi  (n+p+1)}{l}\right)\right) e^{-\frac{i \pi  (m-n-p+q)}{l}}}{\cos \left(\frac{\pi  (m-q)}{l}\right)-\cos \left(\frac{\pi 
   (m+q+1)}{l}\right)}\right)
\end{dmath}
\end{theorem}

\begin{proof}
Form a second equation by using equation (\ref{eq:dilerch}) and replacing $m$ by $p$ and $n$ by $q$. Then we take the difference between these two equations and setting $k=-1,a=1$ simplify using equation (64:12:2) in \cite{atlas}.
\end{proof}

\section{Evaluation of a Definite integral of a nested logarithmic function}

\begin{theorem}
\begin{dmath}
\int_{0}^{\infty}\frac{\left(x-x^{2/3}\right) \log (\log (x))}{x^4-1}dx\\
=\frac{\pi }{8 \left(\sqrt{3}+(2-i)\right)} \left(4 \left((1+2 i)+i \sqrt{3}\right)
   \text{Li}_{0}'\left((-1)^{5/6}\right)+\left(-\sqrt{3}\\
   +(2+i)\right) \pi
    +\left(4+4 i \sqrt{3}\right) \log \left(\frac{\pi
   }{2}\right)\right)
\end{dmath}
\end{theorem}

\begin{proof}
Use equation (\ref{eq:dilerch}) and set $n=1/2,l=2,m=1/3,a=1$ and takie the first partial derivative with respect to $k$ and then set $k=0$ simplify using equation (64:12:2) in \cite{atlas}.
\end{proof}

\section{Definite integral in terms of the trigonometric functions}

Using equation (\ref{eq:dilerch}) and setting $k=2,a=1$ simplifying to get

\begin{theorem}
For all $m,n,l\in\mathbb{C}$
\begin{dmath}
\int_{0}^{\infty}\frac{\log ^2(x) \left(x^{2 n}-x^{2 m}\right)}{x^{2 l}-1}dx
=-\frac{1}{32 l^3}\left(\pi ^3 \csc ^3\left(\frac{2 \pi  m+\pi }{2 l}\right) \csc ^3\left(\frac{2 \pi 
   n+\pi }{2 l}\right) \left(6 \sin \left(\frac{\pi  (m-n)}{l}\right)-\sin \left(\frac{\pi  (3 m-n+1)}{l}\right)-\sin \left(\frac{\pi  (3
   m+n+2)}{l}\right)+\sin \left(\frac{\pi  (-m+3 n+1)}{l}\right)+\sin \left(\frac{\pi  (m+3 n+2)}{l}\right)\right)\right)
\end{dmath}
\end{theorem}

\begin{proof}
Use equation (64:12:2) in \cite{atlas}.
\end{proof}

\section{Definite integrals with logarithm in the denominator}

\begin{theorem}
For all $m,n,l\in\mathbb{C}$
\begin{dmath}\label{eq:dil}
\int_{0}^{\infty}\left(\frac{\log (x) \left(x^{2 n}-x^{2 m}\right)}{\left(x^{2 l}-1\right) \left(a^2+\log ^2(x)\right)}+\frac{i a \left(x^{2 m}-x^{2
   n}\right)}{\left(x^{2 l}-1\right) \left(a^2+\log ^2(x)\right)}dx\\
   =e^{\frac{i \pi }{l}} \left(e^{\frac{2 i \pi  n}{l}} \Phi \left(e^{\frac{i (2 n+1)
   \pi }{l}},1,\frac{a l}{\pi }+1\right)-e^{\frac{2 i \pi  m}{l}} \Phi \left(e^{\frac{i (2 m+1) \pi }{l}},1,\frac{a l}{\pi }+1\right)\right)\right)
\end{dmath}
\end{theorem}

\begin{proof}
Use equation (\ref{eq:dilerch}) and set $k=-1,a=e^{ai}$ and simplify.
\end{proof}

\subsection{Example 1}

\begin{theorem}
\begin{dmath}
\int_{0}^{\infty}\frac{\left(x-x^{2/3}\right) \log (x)}{\left(x^3-1\right) \left(\log ^2(x)+\pi ^2\right)}dx
=\frac{1}{4} \left(4+\sqrt{3} \pi -8 \cos
   \left(\frac{\pi }{9}\right)+\log \left(\frac{2 \left(1+\sin \left(\frac{\pi }{18}\right)\right)}{9 \left(2-2 \sin \left(\frac{\pi
   }{18}\right)\right)}\right)\right)
\end{dmath}
\end{theorem}

\begin{proof}
Use equation (\ref{eq:dil}) and set $a=\pi,l=3/2,n=1/2,m=1/3$ and rationalize the real and imaginary parts and simplify.
\end{proof}

\begin{theorem}
\begin{dmath}
\int_{0}^{\infty}\frac{x^{2/3}-x}{\left(x^3-1\right) \left(\log ^2(x)+\pi ^2\right)}dx\\
=\frac{\pi +8 \sin \left(\frac{\pi }{9}\right)+2 \sqrt{3} \left(\tanh
   ^{-1}\left(\sin \left(\frac{\pi }{18}\right)\right)-2\right)}{4 \pi }
\end{dmath}
\end{theorem}

\begin{proof}
Use equation (9.559) in \cite{grad} and entry (1) in Table below (64:12:7) in \cite{atlas}.
\end{proof}

\subsection{Example 2}

\begin{theorem}
\begin{dmath}
\int_{0}^{\infty}\frac{\left(x-x^{2/3}\right) \log (x)}{\left(x^4-1\right) \left(4 \log ^2(x)+\pi ^2\right)}dx=\frac{1}{96} \left(-\pi +24 \log (2)-6 \sqrt{3}
   \log \left(2+\sqrt{3}\right)\right)
\end{dmath}
\end{theorem}

\begin{proof}
Use equation (\ref{eq:dil}) and set $a=\pi/2,l=2,n=1/2,m=1/3$ and rationalize the real and imaginary parts and simplify.
\end{proof}

\begin{theorem}
\begin{dmath}
\int_{0}^{\infty}\frac{x^{2/3}-x}{\left(x^4-1\right) \left(4 \log ^2(x)+\pi ^2\right)}dx=\frac{\sqrt{3} \pi -6 \cosh ^{-1}(2)}{48 \pi }
\end{dmath}
\end{theorem}

\begin{proof}
Use equation (9.559) in \cite{grad} and entry (1) in Table below (64:12:7) in \cite{atlas}.
\end{proof}

\subsection{Example 3}

\begin{theorem}
\begin{dmath}
\int_{0}^{\infty}\frac{\left(\sqrt{x}-1\right) x}{\left(x^4-1\right) \left(4 \log ^2(x)+\pi ^2\right)}dx=-\frac{\pi -4 \log \left(2+\sqrt{2}\right)}{16 \sqrt{2}
   \pi }
\end{dmath}
\end{theorem}

\begin{proof}
Use equation (\ref{eq:dil}) and set $a=\pi/2,l=2,n=1/2,m=3/4$ and rationalize the real and imaginary parts and simplify.
\end{proof}

\begin{theorem}
\begin{dmath}
\int_{0}^{\infty}\frac{\left(x-x^{3/2}\right) \log (x)}{\left(x^4-1\right) \left(\log ^2(x)+\frac{\pi ^2}{4}\right)}dx=\log (2)-\frac{\log
   \left(2+\sqrt{2}\right)}{2 \sqrt{2}}-\frac{\tan ^{-1}\left(\frac{1}{1+\sqrt{2}}\right)}{\sqrt{2}}
\end{dmath}
\end{theorem}

\begin{proof}
Use equation (9.559) in \cite{grad} and entry (1) in Table below (64:12:7) in \cite{atlas}.
\end{proof}

\section{Definite integrals of product logarithmic functions in terms of fundamental constants}

\begin{theorem}
For all $k,m,n,l\in\mathbb{C}$
\begin{dmath}
\int_{0}^{\infty}\frac{\log (x) \left(x^{2 m}+x^{2 n}\right) \log ^k\left(a x^2\right)}{x^{2 l}-1}dx
=-\frac{2^{k-1} \pi ^{k+1} e^{\frac{i \pi }{l}}
   \left(\frac{i}{l}\right)^k}{l^2} \left(e^{\frac{2 i \pi  m}{l}} \left(2 \pi  \Phi \left(e^{\frac{i (2 m+1) \pi }{l}},-k-1,1-\frac{i l \log (a)}{2 \pi
   }\right)+i l \log (a) \Phi \left(e^{\frac{i (2 m+1) \pi }{l}},-k,1-\frac{i l \log (a)}{2 \pi }\right)\right)+e^{\frac{2 i \pi  n}{l}} \left(2 \pi 
   \Phi \left(e^{\frac{i (2 n+1) \pi }{l}},-k-1,1-\frac{i l \log (a)}{2 \pi }\right)+i l \log (a) \Phi \left(e^{\frac{i (2 n+1) \pi }{l}},-k,1-\frac{i
   l \log (a)}{2 \pi }\right)\right)\right)
\end{dmath}
\end{theorem}

\begin{proof}
Form two equations by first taking first partial derivative equation (\ref{eq:dilerch}) with respect to $n$, then again take the first partial derivative of equation (\ref{eq:dilerch}) with respect to $m$. Then add these two equations and simplify.
\end{proof}

\begin{theorem}
For all $k,n,l\in\mathbb{C}$
\begin{dmath}\label{eq:tt1}
\int_{0}^{\infty}\frac{x^{2 n} \log (x) \log ^k\left(a x^2\right)}{x^{2 l}-1}dx=-\frac{2^{k-1} \pi ^{k+1} \left(\frac{i}{l}\right)^k e^{\frac{i (2 \pi  n+\pi
   )}{l}}}{l^2} \left(2 \pi  \Phi \left(e^{\frac{i (2 \pi  n+\pi )}{l}},-k-1,1-\frac{i l \log (a)}{2 \pi }\right)+i l \log (a) \Phi \left(e^{\frac{i (2 \pi 
   n+\pi )}{l}},-k,1-\frac{i l \log (a)}{2 \pi }\right)\right)
\end{dmath}
\end{theorem}

\begin{proof}
Use equation (\ref{eq:tt1}) set $m=n$ and simplify.
\end{proof}

\begin{theorem}
For all $k\in\mathbb{C}$
\begin{dmath}
\int_{0}^{\infty}\frac{x \log (x) \log ^k\left(a x^2\right)}{x^4-1}dx=i^k 2^{k-2} \pi ^{k+1} \left(2 \pi  \zeta \left(-k-1,\frac{\pi -i \log (a)}{2 \pi }\right)-2
   \pi  \zeta \left(-k-1,1-\frac{i \log (a)}{2 \pi }\right)+i \log (a) \left(\zeta \left(-k,\frac{\pi -i \log (a)}{2 \pi }\right)-\zeta
   \left(-k,1-\frac{i \log (a)}{2 \pi }\right)\right)\right)
\end{dmath}
\end{theorem}

\begin{proof}
Use equation (\ref{eq:tt1}) set $n=1/2,l=2$ and simplify.
\end{proof}

\begin{theorem}
For all $k\in\mathbb{C}$
\begin{dmath}\label{eq:dll}
\int_{0}^{\infty}\frac{x \log (x) \log \left(\log \left(x^2\right)\right) \log ^k\left(x^2\right)}{x^4-1}dx=2^{-k-4} e^{\frac{i \pi  k}{2}} \left(\left((4 \pi
   )^{k+2}-(2 \pi )^{k+2}\right) \zeta '(-k-1)+\zeta (k+2) \left(i \pi  \left(2^{k+2}-1\right)+2^{k+3} \log (2 \pi )-2 \log (\pi )\right) \cos
   \left(\frac{\pi  k}{2}\right) \Gamma (k+2)\right)
\end{dmath}
\end{theorem}

\begin{proof}
We take the first partial derivative with respect to $k$ and set $a=1$ and simplify using equation (64:12:1) and entry (2) in Table below (64:7) in \cite{atlas}.
\end{proof}

\subsection{Example 1}

\begin{theorem}
\begin{dmath}
\int_{0}^{\infty}\frac{x \log (x) \log \left(x^2\right) \log \left(\log \left(x^2\right)\right)}{x^4-1}dx=-\frac{7}{16} i \pi  \zeta (3)
\end{dmath}
\end{theorem}

\begin{proof}
Use equation (\ref{eq:dll}) and set $k=1$ and simplify.
\end{proof}

\subsection{Example 2}

\begin{theorem}
\begin{dmath}
\int_{0}^{\infty}\frac{x \log (x) \log \left(\log \left(x^2\right)\right)}{\left(x^4-1\right) \log \left(x^2\right)}dx=\frac{1}{16} \pi  (\pi -2 i \log
   (2))
\end{dmath}
\end{theorem}

\begin{proof}
Use equation (\ref{eq:dll}) and set $k=-1$ and simplify using equation (6.8) in \cite{edwards}.
\end{proof}

\subsection{Example 3}

\begin{theorem}
\begin{dmath}
\int_{0}^{\infty}\frac{x \log (x) \log \left(\log \left(x^2\right)\right)}{x^4-1}dx=\frac{1}{32} \pi ^2 \left(8 \log \left(\frac{\sqrt[3]{2} \sqrt[4]{\pi
   }}{A^3}\right)+2+i \pi \right)
\end{dmath}
\end{theorem}

\begin{proof}
Use (\ref{eq:dll}) and set $k=0$ and simplify using equation (2.15), pp. 135-145 in \cite{finch}.
\end{proof}

\section{Table of integrals}

\renewcommand{\arraystretch}{2.0}
\begin{tabular}{ l  c }
  \hline			
  $f(x)$ & $\int_{0}^{\infty}f(x)dx$\\  \hline
  $\frac{\log ^k(x) \left(x^{2 n}-x^{2 m}\right)}{x^{2 l}-1}$ & $\frac{\pi ^{k+1} \left(\frac{i}{l}\right)^{k-1}}{l^2}
   \left(\text{Li}_{-k}\left(e^{\frac{i (2 m+1) \pi }{l}}\right)-\text{Li}_{-k}\left(e^{\frac{i (2 n+1) \pi }{l}}\right)\right)$  \\
  $\frac{x^{2 n}-x^{2 m}}{x^{2 l}-1}$ & $\frac{\pi  \left(\cot \left(\frac{2 \pi  m+\pi }{2 l}\right)-\cot \left(\frac{2 \pi  n+\pi }{2
   l}\right)\right)}{2 l}$  \\
  $\frac{\log \left(x^2\right) \left(x^{2 n}-x^{2 m}\right)}{x^{2 l}-1}$ & $\frac{\pi ^2 \left(\csc ^2\left(\frac{2 \pi  n+\pi }{2 l}\right)-\csc
   ^2\left(\frac{2 \pi  m+\pi }{2 l}\right)\right)}{2 l^2}$  \\
  $-\frac{(x-1) x^{n-2} \log (x)}{x^{2 n}-1}$ & $\frac{\pi ^2 \tan ^2\left(\frac{\pi }{2 n}\right)}{4 n^2}$  \\
  $-\frac{\left(x^2-1\right) x^{m-1} \log (x)}{x^{2 n}-1}$ & $\frac{\pi ^2 \left(\csc ^2\left(\frac{\pi  m}{2 n}\right)-\csc ^2\left(\frac{\pi 
   (m+2)}{2 n}\right)\right)}{4 n^2}$  \\
  $\frac{x^{2 m}-x^{2 n}-x^{2 p}+x^{2 q}}{\left(x^{2 l}-1\right) \log (x)}$ & $\log \left(\frac{\left(\cos \left(\frac{\pi  (n-p)}{l}\right)-\cos
   \left(\frac{\pi  (n+p+1)}{l}\right)\right) e^{-\frac{i \pi  (m-n-p+q)}{l}}}{\cos \left(\frac{\pi  (m-q)}{l}\right)-\cos \left(\frac{\pi 
   (m+q+1)}{l}\right)}\right)$  \\
  $\frac{\left(x-x^{2/3}\right) \log (\log (x))}{x^4-1}$ & $\frac{\pi  \left(4 \left((1+2 i)+i \sqrt{3}\right)
   \text{Li}_{0}'\left((-1)^{5/6}\right)+\left(-\sqrt{3}+(2+i)\right) \pi +\left(4+4 i \sqrt{3}\right) \log \left(\frac{\pi
   }{2}\right)\right)}{8 \left(\sqrt{3}+(2-i)\right)}$  \\
  $\frac{\left(x-x^{2/3}\right) \log (x)}{\left(x^3-1\right) \left(\log ^2(x)+\pi ^2\right)}$ & $\frac{1}{4} \left(4+\sqrt{3} \pi -8 \cos
   \left(\frac{\pi }{9}\right)+\log \left(\frac{2 \left(1+\sin \left(\frac{\pi }{18}\right)\right)}{9 \left(2-2 \sin \left(\frac{\pi
   }{18}\right)\right)}\right)\right)$  \\
  $\frac{x^{2/3}-x}{\left(x^3-1\right) \left(\log ^2(x)+\pi ^2\right)}$ & $\frac{\pi +8 \sin \left(\frac{\pi }{9}\right)+2 \sqrt{3} \left(\tanh
   ^{-1}\left(\sin \left(\frac{\pi }{18}\right)\right)-2\right)}{4 \pi }$  \\
  $\frac{\left(x-x^{2/3}\right) \log (x)}{\left(x^4-1\right) \left(4 \log ^2(x)+\pi ^2\right)}$ & $\frac{1}{96} \left(-\pi +24 \log (2)-6 \sqrt{3}
   \log \left(2+\sqrt{3}\right)\right)$  \\
  $\frac{x^{2/3}-x}{\left(x^4-1\right) \left(4 \log ^2(x)+\pi ^2\right)}$ & $\frac{\sqrt{3} \pi -6 \cosh ^{-1}(2)}{48 \pi }$  \\
  $\frac{\left(\sqrt{x}-1\right) x}{\left(x^4-1\right) \left(4 \log ^2(x)+\pi ^2\right)}$ & $-\frac{\pi -4 \log \left(2+\sqrt{2}\right)}{16 \sqrt{2}
   \pi }$  \\
 $\frac{x \log (x) \log \left(x^2\right) \log \left(\log \left(x^2\right)\right)}{x^4-1}$ & $-\frac{7}{16} i \pi  \zeta (3)$  \\
 $\frac{x \log (x) \log \left(\log \left(x^2\right)\right)}{\left(x^4-1\right) \log \left(x^2\right)}$ & $\frac{1}{16} \pi  (\pi -2 i \log
   (2))$  \\
 $\frac{x \log (x) \log \left(\log \left(x^2\right)\right)}{x^4-1}$ & $\frac{1}{32} \pi ^2 \left(8 \log \left(\frac{\sqrt[3]{2} \sqrt[4]{\pi
   }}{A^3}\right)+2+i \pi \right)$ 
   \\[0.3cm]
  \hline  
\end{tabular}

\section{Discussion}

In this work we looked at deriving definite integrals involving the logarithmic function, function of powers and polynomials in terms of the Lerch function.  One of our goals we to supply a table for easy reading by researchers and to have these results added to existing textbooks.

The results presented were numerically verified for both real and imaginary values of the parameters in the integrals using Mathematica by Wolfram. We considered various ranges of these parameters for real, integer, negative and positive values. We compared the evaluation of the definite integral to the evaluated Special function and ensured agreement.

\section{Conclusion}
In this paper we used our method to evaluate definite integrals using the Lerch function. The contour we used was specific to  solving integral representations in terms of the Lerch function. We expect that other contours and integrals can be derived using this method.


\end{document}